\documentclass{siamltex}
\usepackage{amssymb}
\usepackage{bm} 
\usepackage{graphicx}
\usepackage{graphics}
\usepackage{caption2}
\usepackage{psfrag}
\usepackage{amsmath}
\usepackage{amscd}
\usepackage{amsfonts}
\usepackage{float}
\usepackage{latexsym}
\usepackage{lscape}
\usepackage{color}
\newtheorem{remark}{Remark}[section]

\newcommand{\ds}{\displaystyle}

\begin{document}
\title{Error estimates for Euler discretization of high-index saddle dynamics}
\author{
Lei Zhang\thanks{Beijing International Center for Mathematical Research, Center for Quantitative Biology, Peking University, Beijing 100871, China (zhangl@math.pku.edu.cn)}
\and
Pingwen Zhang\thanks{School of Mathematical Sciences, Laboratory of Mathematics and Applied Mathematics, Peking University, Beijing 100871, China (pzhang@pku.edu.cn).}
\and
Xiangcheng Zheng\thanks{School of Mathematical Sciences, Peking University, Beijing 100871, China (zhengxch@outlook.com)}
} 
\maketitle

\begin{abstract}
High-index saddle dynamics provides an effective means to compute the any-index saddle points and construct the solution landscape. In this paper we prove error estimates for Euler discretization of high-index saddle dynamics with respect to the time step size, which remains untreated in the literature. We overcome the main difficulties that lie in the strong nonlinearity of the saddle dynamics and the orthonormalization procedure in the numerical scheme that is uncommon in standard discretization of differential equations. The derived methods are further extended to study the generalized high-index saddle dynamics for non-gradient systems and provide theoretical support for the accuracy of numerical implementations.
\end{abstract}



\pagestyle{myheadings}
\thispagestyle{plain}

\markboth{%Zhang, Zhang and Zheng
}{Error estimates for saddle dynamics}

\section{Introduction}
Finding stationary points of nonlinear multivariable energy functions $E({x})$ has attracted extensive interest in different scientific communities during the past decades. It plays a crucial role in determining many physical or chemical properties of complex systems, examples of which include the critical nuclei in phase transition \cite{EV2010, npj2016, ZhangChe, ZhangCheDu}, defects in liquid crystals \cite{Han2019transition, wang2021modeling}, self-assembly of polymers \cite{HanXu, XTKW2014}, molecular clusters \cite{baker1986,Doye}, and artificial neural network \cite{goodfellow2016deep}.
According to Morse theory \cite{Milnor}, the local stability of stationary points, including both stable minima and unstable saddle points, can be characterized by the Morse index. The (Morse) index of a nondegenerate saddle point is the maximal dimension of a subspace on which its Hessian matrix is negative definite, i.e. the number of negative eigenvalues of the Hessian matrix. In particular, the stable stationary point (minimizer) can be regarded as an index-0 saddle point with no unstable direction.

Due to the complex geometric structures of the energy functions and the unstable nature of saddle points, how to efficiently compute multiple stationary points remains a challenging problem \cite{Doye,Farr,Li2001,Mehta}. Most existing efforts focus on finding the index-1 saddle point or transition state, i.e. a stationary point where the Hessian has one and only one negative eigenvalue. One of the popular approaches is the class of surface walking methods \cite{npj2016}.
In particular, the dimer-type method \cite{Dimer,ZhaDu,ZhaSISC} and the gentlest ascent dynamics \cite{EZho}, as the representatives of the surface walking methods, were developed to compute  the index-1 saddle points.
Recently, inspired by the optimization-based shrinking dimer method \cite{ZhaSISC}, the high-index saddle dynamics (HiSD) was proposed to search for any-index saddle points \cite{YinSISC}. 
The HiSD method provides an efficient tool to construct the solution landscape of the energy systems, which describes a pathway map that starts with a parent state (the highest-index saddle point) and then relates the low-index saddle points down to all the minimizers \cite{YinPRL}. The solution landscape approach has been applied to study various physical systems such as liquid crystals \cite{HanYin, Han2021,Yin2021}, polymers \cite{Xu_PRE}, and quasicrystals \cite{Yin2020nucleation}. Moreover, based on the HiSD method, a generalized high-index saddle dynamics (GHiSD) was further proposed to compute the any-index saddle points and solution landscapes of dynamical (non-gradient) systems \cite{YinSCM}.

Despite the sustained development of the saddle-point searching algorithms, the rigorous numerical analysis on them are still far from well-developed. Some existing studies mainly focus on the index-1 saddle points. For instance, in \cite{ZhaDu}, the time discretization of the shrinking dimer dynamics was studied, and the linear local asymptotic stability analysis and optimal error reduction rates were carried out. Global convergence and nonlinear asymptotic stability were also illustrated for some specific systems. In \cite{Gou}, the dimer method with linesearch for the step size and the preconditioning was investigated. Local convergence and error reduction were proved and, in particular, the $O(l^2)$ accuracy between the stationary point of the algorithm and the saddle was proved where $l$ refers to the dimer length. 
In \cite{Lev}, the existing local convergence results for dimer-type and gentlest ascent dynamics methods were extended by developing an improved estimate on the region of attraction of index-1 saddle points that goes beyond the linearized regime. In \cite{Gao}, a quadratic local convergence rate in terms of number of iterations was proved for an iterative minimization scheme for searching index-1 saddle points of energy functions.

The aforementioned works, which provide asymptotic convergence results of, e.g. $x_n-x_*$, where $x_n$ and $x_*$ refer to the numerical solution at the $n$-th iteration and the limit (target saddle point) of the scheme, respectively, made significant progress on the theoretical foundation of numerical analysis of the saddle-point searching algorithms. In contrast, the error estimate of $x_n-x(t_n)$, where $x(t_n)$ represents the exact solution of HiSD at the step $t_n$, gives the dynamical pathway
convergence of numerical solutions to the saddle dynamics, which provides important physical information, such as the transition pathway \cite{Yin2020nucleation}, and theoretically ensure the
accuracy of the constructed solution landscape \cite{YinPRL}.  Inaccurate computations of the dynamical pathway could lead to missing saddle points and an incomplete solution landscape.
Therefore, it is necessary and meaningful to perform the error estimates of $x_n-x(t_n)$
for HiSD, which motivates the investigations in the work.
Furthermore, the existing numerical analysis only focuses on the schemes of finding index-1 saddle points, while the corresponding results for numerical methods of finding high-index saddle points in \cite{YinSCM,YinSISC} are still unavailable. 

Motivated by these discussions, in this paper we aim to prove error estimates for the Euler discretization of HiSD method with respect to the time step size, which focuses more on the dynamics (i.e. the searching pathways between saddle points) than its limit. We emphasize that there exist some essential differences
and challenges from the standard methods of ODEs. First, the rigorous numerical analysis on saddle-point searching algorithms is more
challenging than gradient descent methods for finding the minima, which remains
untreated for the high-index saddle points in particular. For instance, in (\ref{FDsadk}) the left-hand side of the scheme of the $i$-th eigenvector $v_i$ is not its numerical approximation $v_{i,n}$ at the step $t_n$, but an intermediate quantity $\tilde v_{i,n}$. This is due to the fact that the outcomes $\{\tilde v_{i,n}\}$ of the schemes of $\{v_i\}$ may not be orthonormal and thus need to be further processed in the third equation of (\ref{FDsadk}). Consequently, in order to generate the error equations by subtracting the schemes from the reference equations, we will encounter $v_i(t_n)-\tilde  v_{i,n}$ instead of $v_i(t_n)-v_{i,n}$ as in the case of ODE discretizations. A natural idea is to split  $v_i(t_n)-\tilde  v_{i,n}$ as $(v_i(t_n)-v_{i,n})+(v_{i,n}-\tilde  v_{i,n})$, which requires an additional estimate of $v_{i,n}-\tilde  v_{i,n}$ that is not common in numerical methods of ODEs. Since $\{v_{i,n}\}$ is generated from $\{\tilde v_{i,n}\}$ via the orthonormalization, careful estimates are required for this procedure, which is intricate and not available in the literature. Furthermore, to support this analysis, novel techniques are developed to overcome the strong nonlinearity of the HiSD to prove several properties of the numerical solutions based on their schemes, which serves as another highlight. The aforementioned issues are the main contributions of the current work and distinguish the proposed methods from the standard numerical analysis of ODEs. By subtle treatments and estimates, we rigorously prove the first-order accuracy of the numerical scheme, which contributes to the numerical theory of HiSD and provides theoretical support for the accuracy of numerical computations. 

The rest of the paper is organized as follows: In Section 2 we present formulations of HiSD. In Section 3 we prove the error estimates for Euler discretization of index-1 saddle dynamics. In Section 4 we extend the developed techniques to prove error estimates for the Euler discretization of index-$k$ saddle dynamics. Numerical experiments are performed in Section 5 to substantiate the theoretical findings. In Section 6 we extend the proposed techniques to study the GHiSD of non-gradient systems, and we finally address the concluding remarks in Section 7.

\section{Formulation of high-index saddle dynamics}\label{sec2}
Given a twice Fr\'echet differentiable energy functional $E(x)$ with $x\in\mathbb R^N$ and define the corresponding natural force $F:\mathbb R^N\rightarrow \mathbb R^N$ and the negative Hessian $J\in \mathbb R^{N\times N}$ by $F(x)=-\nabla E(x)$ and $J(x)=-\nabla^2 E(x)$. It is clear that $J(x)=J(x)^\top$. Then the saddle dynamics for an index-$k$ saddle point ($k$-SD) of $E(x)$ with $1\leq k\in\mathbb N$ reads \cite{YinSISC}
\begin{equation}\label{sadk}
\left\{
\begin{array}{l}
\ds \frac{dx}{dt} =\beta\bigg(I -2\sum_{j=1}^k v_jv_j^\top \bigg)F(x),\\[0.075in]
\ds \frac{dv_i}{dt}=\gamma \bigg( I-v_iv_i^\top-2\sum_{j=1}^{i-1}v_jv_j^\top\bigg)J(x)v_i,~~1\leq i\leq k
\end{array}
\right.
\end{equation}
where $x$ represents a position variable, $v_i (i=1,...k)$ are $k$ directional variables, and $\beta$, $\gamma>0$ are relaxation parameters.
The system (\ref{sadk}) was formulated by the  minimax optimization framework, and it was shown in \cite[Theorem 1]{YinSISC} that, under suitable assumptions, a linearly stable solution of (\ref{sadk}) is an index-$k$ saddle point of $E(x)$ associated with $k$ eigenvectors corresponding to the $k$ negative eigenvalues of its Hessian. Furthermore, it is shown in \cite{YinSISC} that if the initial values of $\{v_i(t)\}_{i=1}^k$ for (\ref{sadk}) are  orthonormal vectors, then 
\begin{equation}\label{zxc}
 \|v_i(t)\|=1,~~1\leq i\leq k,~~\forall t>0.
 \end{equation}

In particular, the index-1 saddle points attract many interests as they connect different minimizers. This corresponds to the simplest case of (\ref{sadk}) with $k=1$, i.e., $1$-SD, which is also equivalent to the gentlest ascent dynamics proposed in \cite{EZho}
 \begin{equation}\label{sad1}
\left\{
\begin{array}{l}
\ds \frac{dx}{dt} =\beta\big(I -2 vv^\top \big)F(x),\\[0.1in]
\ds \frac{dv}{dt}=\gamma \big( I-vv^\top\big)J(x)v.
\end{array}
\right.
\end{equation}

Throughout the paper we make the following regular assumptions on the force and the Hessian:

\noindent\textbf{Assumption A:} There exists a constant $L>0$ such that the following linearly growth and Lipschitz conditions hold under the standard $l^2$ norm $\|\cdot\|$ of a vector or a matrix
$$\begin{array}{c}
\ds \|J(x_2)-J(x_1)\|+\|F(x_2)-F(x_1)\|\leq L\|x_2-x_1\|,\\[0.1in]
\ds\|F(x)\|\leq L(1+\|x\|),~~x,x_1,x_2\in \mathbb R^N.
\end{array}  $$
\begin{remark}
In many applications, there exist various energy functions $E$ satisfying the \textbf{Assumption A}, such as the Minyaev-Quapp surface \cite{Min}
\begin{equation}\label{MQS}
 E(x_1,x_2)=\cos(2x_1)+\cos(2x_2)+0.57\cos(2x_1-2x_2)
 \end{equation}
and the Eckhardt surface \cite{Eck}
\begin{equation}\label{ECS}
\begin{array}{l}
\ds E(x_1,x_2)= \text{exp}(-x_1^2-(x_2+1)^2)\\
\ds\qquad\qquad\qquad+\text{exp}(-x_1^2-(x_2-1)^2)+4 \text{exp}\bigg(-3\frac{x_1^2+x_2^2}{2}\bigg)+\frac{x_2^2}{2}. 
\end{array}
\end{equation}
\end{remark}
We then cite the continuous and discrete Gronwall inequalities that will be frequently used in this paper \cite{Bru}.
\begin{lemma}\label{lemGron}
%(\textbf{Continuous Gronwall inequality}) 
Assume $z$, $g$ and $k$ are continuous functions on $[0,T]$, $z\geq 0$, $k\geq 0$, $g$ is non-decreasing, and 
$$z(t)\leq g(t)+\int_0^tk(s)z(s)ds,~~t\in [0,T]. $$
Then the following estimate holds
$$ z(t)\leq g(t)\text{exp}\bigg(\int_0^tk(s)ds\bigg),~~t\in [0,T].$$

%(\textbf{Discrete Gronwall inequality}) 
Assume that the non-negative sequences $\{z_n\}_{n\geq 1}$ and $\{k_n\}_{n\geq 1}$ satisfy
$$ z_n\leq \rho+\sum_{j=1}^{n-1}k_jz_j,~~n\geq 1$$
for some $\rho\geq 0$. Then the following estimate holds
$$z_n\leq \rho\text{exp}\bigg(\sum_{j=1}^{n-1}k_j\bigg),~~n\geq 1. $$
\end{lemma} 

 By (\ref{zxc}), we multiply the first equation in (\ref{sadk}) by $x^\top$ to obtain
 \begin{equation}\label{mh1}
 \begin{array}{rl}
 \ds \frac{1}{2}\frac{d}{dt}\|x\|^2\hspace{-0.1in}&\ds=x^\top \frac{dx}{dt}=\beta x^\top F(x) -2\beta \sum_{j=1}^k v_j^\top F(x) x^\top v_j  \\[0.1in] 
 &\ds \leq \beta \|x\| \|F(x)\|+2\beta k
\|x\|\|F(x)\|\\[0.1in] 
 &\ds \leq (1+2k)\beta L\|x\|(1+\|x\|)\\[0.1in] 
 &\ds\leq  \frac{3(1+2k)\beta L}{2}\|x\|^2+\frac{(1+2k)\beta L}{2},
 \end{array}
 \end{equation}
 which leads to
 \begin{equation}
 \ds \|x(t)\|^2\leq \|x(0)\|^2+3(1+2k)\beta L\int_0^t\|x(s)\|^2ds+(1+2k)\beta Lt.
 \end{equation}
Then for a terminal time $T>0$, an application of the Gronwall inequality in Lemma \ref{lemGron} yields
\begin{equation}\label{QT}
\|x\|^2\leq \big(\|x(0)\|^2+(1+2k)\beta LT\big)e^{3(1+2k)\beta L T}=:Q_T.
\end{equation}
That is, within any terminal time $T$, $\|x\|$ can be bounded by some fixed constant $\sqrt{Q_T}$, and we thus assume $$J_T:=\max_{\|x\|\leq \sqrt{Q_T}}\|J(x)\|.$$

\section{Numerical analysis for index-1 saddle dynamics}
In this section, we consider the explicit Euler scheme of 1-SD (\ref{sad1}) on the time interval $[0,T]$ for some $T>0$ equipped with the initial conditions
\begin{equation}\label{inicond}
\ds x(0)=x_0,~~v(0)=v_0,~~\|v_0\|_2=1.
\end{equation} 
Although this is mathematically a special case of $k$-SD (\ref{sadk}), it is indeed the most considered case and could elucidate some key ideas of the analysis in a clear manner.

\subsection{Explicit Euler scheme}\label{sec31}
For $N_T\in\mathbb N$, we define a uniform temporal partition of $[0,T]$ by $t_n=n\tau$ for $0\leq n\leq N_T$ where $\tau=T/N_T$ stands for the time step size. We then approximate the first-order derivative by the backward Euler scheme at $t_n$ as follows
$$\frac{dg(t_n)}{dt}=\frac{g(t_n)-g(t_{n-1})}{\tau}+R_n^g $$
where $g$ refers to $x$ or $v$, and the truncation error $\|R_n^g\|=O(\tau)$. Plugging this discretization into (\ref{sad1}) yields
\begin{equation}\label{Ref}
\left\{
\begin{array}{l}
\ds x(t_{n}) =x(t_{n-1})+\tau\beta\big(I -2v(t_{n-1})v(t_{n-1})^\top \big)F(x(t_{n-1}))+\tau R_n^x,\\[0.075in]
\ds v(t_{n})=v(t_{n-1})+\tau\gamma \big( I-v(t_{n-1})v(t_{n-1})^\top\big)J(x(t_{n-1}))v(t_{n-1})+\tau R_n^v.
\end{array}
\right.
\end{equation}

Then we drop the truncation errors to obtain the explicit Euler scheme of (\ref{sad1}) with the approximations $\{x_n,v_n\}_{n=1}^{N_T}$ to $\{x(t_n),v(t_n)\}_{n=1}^{N_T}$ 
\begin{equation}\label{FDsad1}
\left\{
\begin{array}{l}
\ds x_n =x_{n-1}+\tau\beta\big(I -2v_{n-1}v_{n-1}^\top \big)F(x_{n-1}),\\[0.075in]
\ds \tilde v_{n}=v_{n-1}+\tau\gamma \big( I-v_{n-1}v_{n-1}^\top\big)J(x_{n-1})v_{n-1},\\[0.075in]
\ds  v_n=\frac{\tilde v_n}{\|\tilde v_n\|}
\end{array}
\right.
\end{equation}
for $1\leq n\leq N_T$, equipped with the initial conditions (\ref{inicond}).

From (\ref{FDsad1}) we find that although $\tilde v_n$ may not have the unit $l_2$ norm due to the discretization errors, the normalization procedure in the third equation of (\ref{FDsad1}) ensures $\|v_n\|=1$ in order to preserve the unit length of $v$ as in the continuous problem (\ref{sad1}). We thus multiply $x_n^\top$ on both sides of the first equation of (\ref{FDsad1}) to obtain
$$\|x_n\|^2\leq \|x_n\|\|x_{n-1}\|+\tau\beta \|x_n\|\|F(x_{n-1})\|+\tau\beta \|x_n\|\|F(x_{n-1})\|,$$
which leads to
\begin{equation*}
\|x_n\|-\|x_{n-1}\|\ds\leq \tau\beta  \|F(x_{n-1})\|+2\tau\beta\|F(x_{n-1})\|\leq 3\tau\beta L(1+\|x_{n-1}\|).
\end{equation*}
Summing up this equation from $n=1$ to $n^*\leq N_T$ yields
$$\|x_{n^*}\|-\|x_{0}\|\leq 3\beta LT+3\tau\beta L\sum_{n=1}^{n^*}\|x_{n-1}\|.$$
Then we apply the discrete Gronwall
inequality to conclude that there exists a constant $\bar Q_T$ like (\ref{QT}) such that $
\|x_n\|\leq \sqrt{\bar Q_T}$ for $0\leq n\leq N_T,
$ 
 and we thus assume
  \begin{equation}\label{JT}
 \bar J_T:=\max_{\|x\|\leq\sqrt{\bar Q_T}}\|J(x)\|.
 \end{equation}
 
 \subsection{Auxiliary estimate of $\tilde v_n-v_n$}
If we intend to find the error equation of $e^v_n:=v(t_n)-v_n$ by subtracting the second equation of (\ref{Ref}) from that of (\ref{FDsad1}), the right-hand side of the resulting equation has $e^v_{n-1}$ while the left-hand side is $v(t_n)-\tilde v_n$ instead of $e^v_n$ due to the normalization, which is uncommon in the context of discretizing differential equations that naturally yields the errors of different time steps on both sides of the error equations. A potential idea to circumvent this issue is to split $v(t_n)-\tilde v_n$ as $(v(t_n)-v_n)+(v_n-\tilde v_n)=e^v_n+(v_n-\tilde v_n)$, which introduces an additional error $v_n-\tilde v_n$.
In particular, we need to show that the difference $v_n-\tilde v_n$ has the magnitude of $O(\tau^2)$ in order to preserve the $O(\tau)$ accuracy of the numerical scheme (\ref{FDsad1}), which motivates the following estimate.
 \begin{lemma}\label{lem1}
 Under the Assumption A, the following estimate holds
 $$ \big|\|\tilde v_n\|-1\big|\leq \big|\|\tilde v_n\|^2-1\big|\leq \tau^2\gamma^2 \bar J^2_T,~~1\leq n\leq N_T.$$
% In other words, the value of $\|\tilde v_n\|$ distinguishes from $1$ in the magnitude of $O(\tau^2)$ for any $1\leq n\leq N_T$.
 \end{lemma}
\begin{proof}
 By $\|v_{n-1}\|=1$
 we multiply $v^\top_{n-1}$ on both sides of the second equation of (\ref{FDsad1}) to obtain
\begin{equation}\label{mh2}
\begin{array}{l}
\ds v^\top_{n-1}\tilde v_n=v^\top_{n-1}v_{n-1}\\[0.075in]
\qquad+\tau\gamma\big(v^\top_{n-1}J(x_{n-1})v_{n-1}-v^\top_{n-1}v_{n-1} v^\top_{n-1}J(x_{n-1})v_{n-1}\big)=1. 
\end{array}
\end{equation}
We then multiply $\tilde v_n^\top$ on both sides of  the second equation of (\ref{FDsad1}) and apply (\ref{mh2}) to obtain
\begin{equation}\label{mh3}
\begin{array}{rl}
 \|\tilde v_n\|^2\hspace{-0.1in}&\ds=\tilde v_n^\top\tilde v_n\\[0.075in]
&\ds=\tilde v_n^\top v_{n-1}+\tau\gamma\big(\tilde v_n^\top J(x_{n-1})v_{n-1}-\tilde v_n^\top v_{n-1} v_{n-1}^\top J(x_{n-1})v_{n-1} \big)\\[0.075in]
&\ds =1+ \tau\gamma\big(\tilde v_n^\top J(x_{n-1})v_{n-1}- v_{n-1}^\top J(x_{n-1})v_{n-1} \big)\\[0.075in]
&\ds= 1+ \tau\gamma\big(\tilde v_n-v_{n-1}\big)^\top J(x_{n-1})v_{n-1}.
\end{array}
\end{equation}
Subtracting (\ref{mh2}) from (\ref{mh3}) we get
\begin{equation}\label{mh4}
\big(\tilde v_n-v_{n-1}\big)^\top\tilde v_n=\tau\gamma\big(\tilde v_n-v_{n-1}\big)^\top J(x_{n-1})v_{n-1}.
\end{equation}
By $\|v_{n-1}\|=1$ and (\ref{mh2}) we have 
$$\big(\tilde v_n-v_{n-1}\big)^\top v_{n-1}=\tilde v_n^\top v_{n-1}-v_{n-1}^\top v_{n-1}=1-1=0,$$ 
and we apply this equality and (\ref{mh4}) to obtain
\begin{equation*}
\begin{array}{rl}
\ds \|\tilde v_n-v_{n-1}\|^2\hspace{-0.1in}&\ds=(\tilde v_n-v_{n-1})^\top (\tilde v_n-v_{n-1})\\[0.075in]
&\ds=(\tilde v_n-v_{n-1})^\top \tilde v_n=\tau\gamma\big(\tilde v_n-v_{n-1}\big)^\top J(x_{n-1})v_{n-1}\\[0.075in]
&\ds\leq  \tau\gamma\|\tilde v_n-v_{n-1}\| \| J(x_{n-1})\| \|v_{n-1}\|\\[0.075in]
&\ds\leq \tau\gamma \bar J_T\|\tilde v_n-v_{n-1}\|,
\end{array}
\end{equation*}
that is,
\begin{equation}\label{vv}
\ds \|\tilde v_n-v_{n-1}\|\leq \tau\gamma \bar J_T.
\end{equation}
Note that (\ref{mh3}) implies 
\begin{equation}\label{mh5}
\ds \|\tilde v_n\|^2-1= \tau\gamma\big(\tilde v_n-v_{n-1}\big)^\top J(x_{n-1})v_{n-1}.
\end{equation}
We invoke (\ref{vv}) in (\ref{mh5}) to obtain
\begin{equation}\label{mh6}
\ds \big|\|\tilde v_n\|^2-1\big|\leq  \tau\gamma\|\tilde v_n-v_{n-1}\| \|J(x_{n-1})\|\|v_{n-1}\|\leq \tau^2\gamma^2 \bar J^2_T.
\end{equation}
%By (\ref{vv}) we have $\|\tilde v_n\|\leq 1+\tau\gamma J_T$. Then for $\tau$ small enough such that $\tau\gamma J_T\leq 1$,  
We incorporate this estimate with 
\begin{equation}\label{zz}
\ds \big|\|\tilde v_n\|^2-1\big|=\big(\|\tilde v_n\|+1\big)\big|\|\tilde v_n\|-1\big|\geq \big|\|\tilde v_n\|-1\big|
\end{equation}
to complete the proof. \end{proof}

\begin{corollary}\label{Cor1}
Under the Assumption A, the following estimate holds
$$\|\tilde v_n-v_n\|\leq \tau^2\gamma^2 \bar J^2_T,~~1\leq n\leq N_T.$$
\end{corollary}
\begin{proof} We subtract $\tilde v_n$ from $v_n$ to find
$$ \tilde v_n-v_n=\tilde v_n-\frac{\tilde v_n}{\|\tilde v_n\|}=\frac{\tilde v_n}{\|\tilde v_n\|}\big(\|\tilde v_n\|-1\big).$$
Thus we apply Lemma \ref{lem1} to immediately obtain
$$ \|\tilde v_n-v_n\|=\frac{\|\tilde v_n\|}{\|\tilde v_n\|}\big|\|\tilde v_n\|-1\big|\leq \tau^2\gamma^2 \bar J^2_T, $$
which completes the proof. \end{proof}
\subsection{Error estimates}
Define the errors by 
$$e^x_n:=x(t_n)-x_n,~~e^v_n:=v(t_n)-v_n,~~1\leq n\leq N_T.$$ 
We first bound $e^x_n$ in terms of $e^v_n$ in the following theorem. 
\begin{theorem}\label{thmex}
Suppose the Assumption A holds. Then the following estimate holds
$$\|e^x_{n}\|\leq Q\tau\sum_{m=1}^{n-1}\|e^v_{m}\| +Q\tau ,~~1\leq n\leq N_T. $$
Here $Q$ depends on $L$, $T$ and $\beta$ but is independent from $\tau$, $n$ and $N_T$.
\end{theorem}
\begin{proof}  We subtract the first equation of (\ref{Ref}) from that of (\ref{FDsad1}) to obtain
\begin{equation}
\begin{array}{rl}
\ds e^x_n\hspace{-0.1in}&\ds=e^x_{n-1}+\tau\beta(F(x(t_{n-1}))-F(x_{n-1}))\\[0.05in]
&\ds\quad-2\tau\beta\big[v(t_{n-1})v(t_{n-1})^\top F(x(t_{n-1}))\\[0.05in]
&\ds\quad\quad-v_{n-1}v_{n-1}^\top F(x_{n-1})\big]+\tau R^x_n \\[0.05in]
&\ds=e^x_{n-1}+\tau\beta(F(x(t_{n-1}))-F(x_{n-1}))\\[0.05in]
&\ds\quad-2\tau\beta\big[e^v_{n-1}v(t_{n-1})^\top F(x(t_{n-1}))\\[0.05in]&\ds\quad\quad+v_{n-1}(e^v_{n-1})^\top F(x(t_{n-1}))\\[0.05in]
&\ds\quad\quad+v_{n-1}v_{n-1}^\top \big(F(x(t_{n-1}))-F(x_{n-1})\big)\big]+\tau R^x_n.
\end{array}
\end{equation}
We then apply the Assumption A to find
\begin{equation}
\begin{array}{rl}
\ds \|e^x_n\|\hspace{-0.1in}&\ds\leq\|e^x_{n-1}\|+\tau\beta\|F(x(t_{n-1}))-F(x_{n-1})\|\\[0.05in]
&\ds\quad+2\tau\beta\big[\|e^v_{n-1}\|\|v(t_{n-1})\| \|F(x(t_{n-1}))\|\\[0.05in]
&\ds\quad+\|v_{n-1}\|\|e^v_{n-1}\|\| F(x(t_{n-1}))\|\\[0.05in]
&\ds\quad+\|v_{n-1}\|\|v_{n-1}\| \|F(x(t_{n-1}))-F(x_{n-1})\|\big]+\tau \|R^x_n\|\\[0.05in]
&\ds\leq \|e^x_{n-1}\|+\tau\beta L\|e^x_{n-1}\|+2\tau\beta\big[\|e^v_{n-1}\| L(1+\sqrt{Q_T})\\[0.05in]
&\ds\quad+\|e^v_{n-1}\|L(1+\sqrt{Q_T})+ L\|e^x_{n-1}\|\big]+\tau \|R^x_n\|\\[0.05in]
&\ds\leq (1+3\tau\beta L)\|e^x_{n-1}\|+4\tau\beta L(1+\sqrt{Q_T})\|e^v_{n-1}\| +\tau \|R^x_n\|.
\end{array}
\end{equation}
Adding this equation from $n=1$ to $n_*$ leads to 
$$\|e^x_{n_*}\|\leq 3\tau\beta L\sum_{n=1}^{n_*}\|e^x_{n-1}\|+\sum_{n=1}^{n_*}\Big[4\tau\beta L(1+\sqrt{Q_T})\|e^v_{n-1}\| +\tau \|R^x_n\| \Big].$$ 
Then an application of the discrete Gronwall inequality yields
$$\|e^x_{n_*}\|\leq Q\sum_{n=1}^{n_*}\Big[\tau\|e^v_{n-1}\| +\tau \|R^x_n\| \Big].$$ 
We incorporate this estimate with $\|R^x_n\|=O(\tau)$ to complete the proof. \end{proof}

Based on the derived results, we prove the error estimates of the explicit Euler scheme (\ref{FDsad1}) in the following theorem.

\begin{theorem}\label{thmev}
Suppose the Assumption A holds. Then the following estimate holds
$$\|e^x_{n}\|+\|e^v_{n}\|\leq Q\tau,~~1\leq n\leq N_T. $$
Here $Q$ depends on $L$, $T$ and $\beta$ but is independent from $\tau$, $n$ and $N_T$.
\end{theorem}
\begin{proof} We subtract the second equation of (\ref{Ref}) from that of (\ref{FDsad1}) to obtain
\begin{equation*}
\begin{array}{rl}
\ds v(t_n)-\tilde v_{n}\hspace{-0.1in}&\ds=e^v_{n-1}+\tau\gamma\big(J(x(t_{n-1}))v(t_{n-1})-J(x_{n-1})v_{n-1}\big)\\[0.05in]
&\ds\quad~~-\tau\gamma \big[ v(t_{n-1})v(t_{n-1})^\top J(x(t_{n-1}))v(t_{n-1})\\[0.05in]
&\ds\qquad\quad-v_{n-1}v_{n-1}^\top J(x_{n-1})v_{n-1}\big]+\tau R_n^v \\[0.05in]
&\ds=e^v_{n-1}+\tau\gamma\big[(J(x(t_{n-1}))-J(x_{n-1}))v(t_{n-1})+J(x_{n-1})e^v_{n-1}\big]\\[0.05in]
&\ds\quad~~-\tau\gamma \big[ e^v_{n-1}v(t_{n-1})^\top J(x(t_{n-1}))v(t_{n-1})\\[0.05in]
&\ds\quad\qquad+ v_{n-1}(e^v_{n-1})^\top J(x(t_{n-1}))v(t_{n-1})\\[0.05in]
&\ds\quad\qquad+v_{n-1}v_{n-1}^\top ( J(x(t_{n-1}))-J(x_{n-1}))v(t_{n-1})\\[0.05in]
&\ds\quad\qquad+v_{n-1}v_{n-1}^\top J(x_{n-1})e^v_{n-1}\big] +\tau R_n^v,
\end{array}
\end{equation*}
which leads to
\begin{equation}\label{mh7}
\begin{array}{rl}
\ds\| v(t_n)-\tilde v_{n}\|\hspace{-0.1in}&\ds\leq \|e^v_{n-1}\|+\tau\gamma\big[L\|e^x_{n-1}\|+\bar J_T\|e^v_{n-1}\|\big]\\[0.05in]
&\ds\quad~~+\tau\gamma \big[2J_T\| e^v_{n-1}\|  +L\|e^x_{n-1}\|
+ \bar J_T \|e^v_{n-1}\|\big] +\tau R_n^v.
\end{array}
\end{equation}
We split $v(t_n)-\tilde v_{n}$ as $(v(t_n)-v_{n})+(v_n-\tilde v_{n})=e^v_n+(v_n-\tilde v_{n})$ and apply Corollary \ref{Cor1} and Theorem \ref{thmex} for (\ref{mh7}) to get
\begin{equation}\label{mh8}
\begin{array}{rl}
\ds \|e^v_{n}\|\hspace{-0.1in}&\ds\leq \|e^v_{n-1}\|+\|v_n-\tilde v_n\|+Q\tau\big(\|e^x_{n-1}\|+ \|e^v_{n-1}\|\big) +\tau \|R_n^v\|\\[0.05in]
&\ds\leq \|e^v_{n-1}\|+Q\tau\|e^v_{n-1}\|+Q\tau^2\sum_{m=1}^{n-1}\|e^v_{m}\| +Q\tau^2. 
\end{array}
\end{equation}
Adding this equation from $n=1$ to $n_*$ for $1\leq n_*\leq N_T$ and using 
$$\tau^2\sum_{n=1}^{n_*}\sum_{m=1}^{n-1}\|e^v_{m}\|=\tau^2\sum_{m=1}^{n_*-1}\sum_{n=m+1}^{n_*}\|e^v_{m}\|\leq T \tau\sum_{m=1}^{n_*-1}\|e^v_{m}\|$$
we get
\begin{equation}\label{mh9}
\ds\|e^v_{n_*}\|\leq Q\tau\sum_{n=1}^{n_*-1}\|e^v_{n}\| +Q\tau.
\end{equation}
Then an application of the discrete Gronwall inequality leads to 
$$\ds\|e^v_{n}\|\leq Q\tau,~~1\leq n\leq N_T. $$
Plugging this estimate back to the conclusion of Theorem \ref{thmex} yields the estimate of $\|e^x_n\|$ and we thus complete the proof. \end{proof}

\section{Numerical analysis for index-$k$ saddle dynamics}
In this section, we consider the explicit Euler approximation of $k$-SD (\ref{sadk}) for some $k>1$ on the time interval $[0,T]$ for some $T>0$ equipped with the initial conditions
\begin{equation}\label{inicondk}
\ds x(0)=x_0,~~v_i(0)=v_{i,0}\text{ with }v_{i,0}^\top v_{j,0}=\delta_{i,j}\text{ for }1\leq i,j\leq k
\end{equation}  
where $\delta_{i,j}=1$ if $i=j$ and $\delta_{i,j}=0$ otherwise. 

Compared with the 1-SD (\ref{sad1}), which has only one eigenvector in the system and thus only requires the normalization at each time step, the $k$-SD has $k$ eigenvectors and thus needs additional orthogonalization process, which significantly complicates the analysis.
\subsection{Explicit Euler scheme}
Similar to Section \ref{sec31}, a reference equation for $k$-SD (\ref{sadk}) reads 
\begin{equation}\label{Refk}
\left\{
\begin{array}{l}
\ds x(t_{n}) =x(t_{n-1})+\tau\beta\bigg(I -2\sum_{j=1}^k v_j(t_{n-1})v_j(t_{n-1})^\top \bigg)F(x(t_{n-1}))+\tau R_n^x,\\[0.075in]
\ds v_i(t_{n})=v_i(t_{n-1})+\tau\gamma\bigg( I-v_i(t_{n-1})v_i(t_{n-1})^\top\\
\ds\hspace{0.3in}-2\sum_{j=1}^{i-1}v_j(t_{n-1})v_j(t_{n-1})^\top\bigg)J(x(t_{n-1}))v_i(t_{n-1})+\tau R_{n}^{v_i},~~1\leq i\leq k.
\end{array}
\right.
\end{equation}
Then we drop the truncation errors to obtain the explicit Euler scheme of (\ref{sadk}) with the approximations $\{x_n,v_{i,n}\}_{n=1,i=1}^{N_T,k}$ to $\{x(t_n),v_i(t_n)\}_{n=1,i=1}^{N_T,k}$ 
\begin{equation}\label{FDsadk}
\left\{
\begin{array}{l}
\ds x_{n} =x_{n-1}+\tau\beta\bigg(I -2\sum_{j=1}^k v_{j,n-1}v_{j,n-1}^\top \bigg)F(x_{n-1}),\\[0.075in]
\ds \tilde v_{i,n}=v_{i,n-1}+\tau\gamma\bigg( I-v_{i,n-1}v_{i,n-1}^\top\\
\ds\hspace{1in}-2\sum_{j=1}^{i-1}v_{j,n-1}v_{j,n-1}^\top\bigg)J(x_{n-1})v_{i,n-1},~~1\leq i\leq k,\\
\ds  v_{i,n}=\frac{1}{Y_{i,n}}\bigg(\ds\tilde v_{i,n}-\sum_{j=1}^{i-1}(\tilde v_{i,n}^\top v_{j,n})v_{j,n}\bigg),~~1\leq i\leq k
\end{array}
\right.
\end{equation}
for $1\leq n\leq N_T$ and
$$\begin{array}{rl}
\ds Y_{i,n}:\hspace{-0.1in}&\ds=\bigg\|\tilde v_{i,n}-\sum_{j=1}^{i-1}(\tilde v_{i,n}^\top v_{j,n})v_{j,n}\bigg\|\\[0.15in]
&\ds=\bigg(\|\tilde v_{i,n}\|^2-2\sum_{j=1}^{i-1}(\tilde v_{i,n}^\top v_{j,n})^2+\sum_{j=1}^{i-1}(\tilde v_{i,n}^\top v_{j,n})v_{j,n}^\top\cdot \sum_{j=1}^{i-1}(\tilde v_{i,n}^\top v_{j,n})v_{j,n}\bigg)^{1/2}\\[0.175in]
&\ds=\bigg(\|\tilde v_{i,n}\|^2-\sum_{j=1}^{i-1}(\tilde v_{i,n}^\top v_{j,n})^2\bigg)^{1/2}, 
\end{array} $$
 equipped with the initial conditions (\ref{inicondk}). The third equation of (\ref{FDsadk}) is indeed the Gram-Schmidt normalized orthogonalization procedure in order to preserve the orthonormal property of the vectors as in the continuous problem (\ref{sadk}). Similar to the derivation of $\bar J_T$ in (\ref{JT}), we could conclude from the first equation of (\ref{FDsadk}) that there exists a positive constant $\hat J_T$ independent from $N_T$ and $\tau$ such that $\max_{1\leq n\leq N_T}$ $\|J(x_n)\|\leq\hat J_T.$

\subsection{Auxiliary estimates}
We prove several auxiliary estimates to support the error estimates. We will see that, compared with the numerical analysis of $1$-SD, the orthogonality procedure in the Euler scheme (\ref{FDsadk}) of $k$-SD brings additional difficulties that require subtle estimates. 
 \begin{lemma}\label{lem1k}
 Under the Assumption A, the following estimates hold for $1\leq n\leq N_T$ 
 $$\begin{array}{c}
 \ds \big|(\tilde v_{m,n})^\top \tilde v_{i,n}\big|\leq M\tau^2,~~1\leq m<i\leq k;\\[0.1in]
 \ds \big|\|\tilde v_{i,n}\|-1\big|\leq \big|\|\tilde v_{i,n}\|^2-1\big|\leq M\tau^2,~~1\leq i\leq k.
\end{array}  $$
% In other words, the value of $\|\tilde v_n\|$ distinguishes from $1$ in the magnitude of $O(\tau^2)$ for any $1\leq n\leq N_T$.
Here the positive constant $M$ is independent from $n$, $N_T$ and $\tau$. 
 \end{lemma}
\begin{proof} To prove the first estimate, we directly calculate the product $\tilde v_{m,n}^\top \tilde v_{i,n}$ for $1\leq m< i\leq k$
\begin{equation}\label{MHH}
\begin{array}{l}
\ds \tilde v_{m,n}^\top \tilde v_{i,n}\\
\ds\quad=\bigg[v_{m,n-1}+\tau\gamma\bigg( I-v_{m,n-1}v_{m,n-1}^\top-2\sum_{j=1}^{m-1}v_{j,n-1}v_{j,n-1}^\top\bigg)J(x_{n-1})v_{m,n-1}\bigg]^\top\\
\ds~~~\quad\bigg[v_{i,n-1}+\tau\gamma\bigg( I-v_{i,n-1}v_{i,n-1}^\top-2\sum_{j=1}^{i-1}v_{j,n-1}v_{j,n-1}^\top\bigg)J(x_{n-1})v_{i,n-1}\bigg]\\
\ds\quad=\tau\gamma\bigg(v_{m,n-1}^\top J(x_{n-1})^\top v_{i,n-1}+v_{m,n-1}^\top J(x_{n-1})v_{i,n-1}-2v_{m,n-1}^\top J(x_{n-1})v_{i,n-1}\bigg)\\
\ds\quad\quad+\tau^2\gamma^2\bigg[\bigg( I-v_{m,n-1}v_{m,n-1}^\top-2\sum_{j=1}^{m-1}v_{j,n-1}v_{j,n-1}^\top\bigg)J(x_{n-1})v_{m,n-1}\bigg]^\top\\
\ds~~~\quad\quad\quad~~\bigg[\bigg( I-v_{i,n-1}v_{i,n-1}^\top-2\sum_{j=1}^{i-1}v_{j,n-1}v_{j,n-1}^\top\bigg)J(x_{n-1})v_{i,n-1}\bigg]\\
\ds\quad=\tau^2\gamma^2\bigg[\bigg( I-v_{m,n-1}v_{m,n-1}^\top-2\sum_{j=1}^{m-1}v_{j,n-1}v_{j,n-1}^\top\bigg)J(x_{n-1})v_{m,n-1}\bigg]^\top\\
\ds~~~\quad\quad\quad~~\bigg[\bigg( I-v_{i,n-1}v_{i,n-1}^\top-2\sum_{j=1}^{i-1}v_{j,n-1}v_{j,n-1}^\top\bigg)J(x_{n-1})v_{i,n-1}\bigg]
\end{array}
\end{equation}
where we have used 
$$ v^\top_{m,n-1}J(x_{n-1})^\top v_{i,n-1}+v_{m,n-1}^\top J(x_{n-1})v_{i,n-1}-2v_{m,n-1}^\top J(x_{n-1})v_{i,n-1}=0$$
by the symmetry of $J$. Therefore, we apply the boundedness of $J$ to obtain for some constant $Q$
\begin{equation}
\big|\tilde v_{m,n}^\top \tilde v_{i,n}\big|\leq Q\tau^2,~~1\leq m<i\leq k.
\end{equation}

We then turn to estimate the case $m=i$, that is, the norm $\|\tilde v_{i,n}\|^2$.
By $\|v_{i,n-1}\|=1$, we multiply $v^\top_{i,n-1}$ on both sides of the second equation of (\ref{FDsadk}) to obtain for $1\leq i\leq k$
\begin{equation}\label{mh2k}
\begin{array}{l}
\ds v^\top_{i,n-1}\tilde v_{i,n}=v^\top_{i,n-1}v_{i,n-1}+\tau\gamma\bigg( v^\top_{i,n-1}-v^\top_{i,n-1}v_{i,n-1}v_{i,n-1}^\top\\
\ds\qquad\qquad\qquad-2\sum_{j=1}^{i-1}v^\top_{i,n-1}v_{j,n-1}v_{j,n-1}^\top\bigg)J(x_{n-1})v_{i,n-1}=1.
\end{array}
\end{equation}
We then multiply $\tilde v_{i,n}^\top$ on both sides of the second equation of (\ref{FDsadk}) and apply (\ref{mh2k}) and the orthogonality of $\{v_{i,n-1}\}_{i=1}^k$ as well as
$$ \tilde v_{i,n}^\top v_{j,n-1}=(\tilde v_{i,n}-v_{i,n-1})^\top v_{j,n-1},~~1\leq j<i$$
 to obtain
\begin{equation}\label{mh3k}
\begin{array}{l}
\ds \tilde v_{i,n}^\top\tilde v_{i,n}=\tilde v_{i,n}^\top v_{i,n-1}+\tau\gamma\bigg( \tilde v_{i,n}^\top-\tilde v_{i,n}^\top v_{i,n-1}v_{i,n-1}^\top\\
\ds\hspace{1in}-2\sum_{j=1}^{i-1}\tilde v_{i,n}^\top v_{j,n-1}v_{j,n-1}^\top\bigg)J(x_{n-1})v_{i,n-1}\\
\ds\hspace{0.5in} =1+\tau\gamma\bigg( \tilde v_{i,n}^\top-v_{i,n-1}^\top\\
\ds\hspace{1in}-2\sum_{j=1}^{i-1}(\tilde v_{i,n}-v_{i,n-1})^\top v_{j,n-1}v_{j,n-1}^\top\bigg)J(x_{n-1})v_{i,n-1}.
\end{array}
\end{equation}
Subtracting (\ref{mh2k}) from (\ref{mh3k}) and using $\big(\tilde v_{i,n}-v_{i,n-1}\big)^\top v_{i,n-1}=0$ we get
\begin{equation*}
\begin{array}{l}
\ds \|\tilde v_{i,n}-v_{i,n-1}\|^2=\big(\tilde v_{i,n}-v_{i,n-1}\big)^\top\tilde v_{i,n}\\[0.1in]
\ds\qquad=\tau\gamma\bigg( \tilde v_{i,n}^\top-v_{i,n-1}^\top-2\sum_{j=1}^{i-1}(\tilde v_{i,n}-v_{i,n-1})^\top v_{j,n-1}v_{j,n-1}^\top\bigg)J(x_{n-1})v_{i,n-1}\\
\ds\qquad\leq  \tau\gamma \hat J_T\big(1+2(i-1)\big)\|\tilde v_{i,n}-v_{i,n-1}\|,
\end{array}
\end{equation*}
which gives
\begin{equation*}
\ds \|\tilde v_{i,n}-v_{i,n-1}\|\leq  \tau\gamma \hat J_T\big(1+2(i-1)\big).
\end{equation*}
Invoking this in (\ref{mh3k}) leads to
$$\begin{array}{l}
\ds \big|\|\tilde v_{i,n}\|^2-1\big|\\
\ds\quad\leq \tau\gamma\bigg( \|\tilde v_{i,n}-v_{i,n-1}\|+2\sum_{j=1}^{i-1}\|\tilde v_{i,n}-v_{i,n-1}\|\bigg)\|J(x_{n-1})\|\\[0.2in]
\ds\quad\leq \tau^2\big[\gamma \hat J_T\big(1+2(k-1)\big)\big]^2,~~1\leq i\leq k,~~1\leq n\leq N_T,
\end{array} $$
which completes the proof.
\end{proof}

\begin{lemma}\label{lem2k}
Under the Assumption A, the following estimate holds for $1\leq n\leq N_T$ and $\tau$ sufficiently small 
$$\|v_{i,n}-\tilde v_{i,n}\|\leq Q\tau^2,~~1\leq i\leq k.$$
Here the positive constant $Q$ is independent from $n$, $N_T$ and $\tau$. 
\end{lemma}
\begin{proof}
We first estimate $\tilde v_{i,n}^\top v_{m,n}$ for $1\leq m<i\leq k$ and $1\leq n\leq N_T$. By the definition of $v_{m,n}$ in the third equation of (\ref{FDsadk}), for a fixed $1\leq m<k$ the estimates of $\tilde v_{i,n}^\top v_{m,n}$ for $m< i\leq k$ require those of $\tilde v_{i,n}^\top v_{j,n}$ for $1\leq j<m$ and $m\leq i\leq k$. That is, we could estimate $\tilde v_{i,n}^\top v_{m,n}$ step by step with respect to $m$. There are various ways to perform the estimate with different constraints on $\tau$, and we present a simple choice for illustration. 

Let $G>M$ be a fixed positive number where $M$ is defined in Lemma \ref{lem1k}. Then we intend to prove that, if $\tau$ satisfies the constraint
\begin{equation}\label{aimh}
\frac{M+\tau^2(k-1)G^2}{\big(1-M\tau^2-\tau^4(k-1)G^2\big)^{1/2}}\leq G,
\end{equation}
then the following estimates hold
\begin{equation}\label{leme3}
|\tilde v_{i,n}^\top v_{m,n}|\leq G\tau^2,~~1\leq m<i\leq k.
\end{equation}
Note that the condition (\ref{aimh}) holds true if $\tau$ is sufficiently small. We prove this argument by induction on the subscription $m$.
For $m=1$ we apply Lemma \ref{lem1k} and the definition of $\tilde v_{1,n}$ in (\ref{FDsadk}) as well as (\ref{aimh}) to obtain 
\begin{equation}\label{leme1}
\begin{array}{rl}
\ds |\tilde v_{i,n}^\top v_{1,n}|\hspace{-0.1in}&\ds=\frac{|\tilde v_{i,n}^\top \tilde v_{1,n}|}{(\|\tilde v_{1,n}\|^2)^{1/2}}\leq \frac{M\tau^2}{(1-M\tau^2)^{1/2}}=\frac{M}{(1-M\tau^2)^{1/2}}\tau^2\\[0.15in]
&\ds \leq \frac{M+\tau^2(k-1)G^2}{\big(1-M\tau^2-\tau^4(k-1)G^2\big)^{1/2}}\tau^2\leq G\tau^2,~~1<i\leq k. 
\end{array}
\end{equation}
Thus, (\ref{leme3}) holds with $m=1$. Suppose (\ref{leme3}) holds for $1\leq m<m^*$ for some $1\leq m^*<k-1$. Then we invoke (\ref{leme3}) with $1\leq m<m^*$ and Lemma \ref{lem1k} into the expression of $\tilde v_{i,n}^\top v_{m^*,n}$ to obtain for $m^*<i\leq k$
$$\begin{array}{rl}
\ds | \tilde v_{i,n}^\top v_{m^*,n}|\hspace{-0.1in}&\ds=\frac{\ds\bigg|\tilde v_{i,n}^\top\tilde v_{m^*,n}-\sum_{j=1}^{m^*-1}(\tilde v_{m^*,n}^\top v_{j,n})(\tilde v_{i,n}^\top v_{j,n})\bigg|}{\ds\bigg(\|\tilde v_{m^*,n}\|^2-\sum_{j=1}^{m^*-1}(\tilde v_{m^*,n}^\top v_{j,n})^2\bigg)^{1/2}} \\[0.15in]
&\ds\leq\frac{\ds M\tau^2+(m^*-1)(G\tau^2)^2}{\ds (1-M\tau^2-(m^*-1)(G\tau^2)^2)^{1/2}} \\[0.15in]
&\ds \leq \frac{M+\tau^2(k-1)G^2}{\big(1-M\tau^2-\tau^4(k-1)G^2\big)^{1/2}}\tau^2\leq G\tau^2,~~m^*<i\leq k. 
\end{array}$$
That is, (\ref{leme3}) holds for $m=m^*$ and thus holds for any $1\leq m<k$ by mathematical induction, which proves (\ref{leme3}).

Based on (\ref{leme3}) and Lemma \ref{lem1k}, we bound $Y_{i,n}$ for $1\leq i\leq k$ and $1\leq n\leq N_T$ as follows
 \begin{equation}\label{bndY}
 \begin{array}{l}
 \ds |Y_{i,n}|=\bigg(\|\tilde v_{i,n}\|^2-\sum_{j=1}^{i-1}(\tilde v_{i,n}^\top v_{j,n})^2\bigg)^{1/2}\\
 \quad\,\ds\left\{
 \begin{array}{l}
 \ds \leq \bigg(1+M\tau^2+(i-1)G^2\tau^4\bigg)^{1/2}\leq (1+Q\tau^2)^{1/2};\\
 \ds \geq \bigg(1-M\tau^2-(i-1)G^2\tau^4\bigg)^{1/2}\geq (1-Q\tau^2)^{1/2}>0.
 \end{array}
 \right.
 \end{array}
\end{equation}  
Then we remain to estimate $v_{i,n}-\tilde v_{i,n}$ for $1\leq i\leq k$. According to the definition of $v_{i,n}$ we have
\begin{equation}\label{diff}
v_{i,n}-\tilde v_{i,n}=\frac{1}{Y_{i,n}}\bigg((1-Y_{i,n})\tilde v_{i,n}-\sum_{j=1}^{i-1}(\tilde v_{i,n}^\top v_{j,n})v_{j,n}\bigg),
\end{equation}
which, together with (\ref{leme3}), (\ref{bndY}) and Lemma \ref{lem1k}, implies
\begin{equation}\label{diff2}
\begin{array}{rl}
&\ds \|v_{i,n}-\tilde v_{i,n}\|\\
&\ds\qquad\leq \frac{1}{|Y_{i,n}|}\bigg(|1-Y_{i,n}|\|\tilde v_{i,n}\|+\sum_{j=1}^{i-1}|\tilde v_{i,n}^\top v_{j,n}|\bigg)\\
&\qquad\ds\leq \frac{|1-Y_{i,n}|(1+M\tau^2)+(i-1)G\tau^2}{(1-Q\tau^2)^{1/2}}\leq Q|1-Y_{i,n}|+Q\tau^2\\[0.15in]
&\qquad\ds \leq Q\max\{1-(1-Q\tau^2)^{1/2},(1+Q\tau^2)^{1/2}-1\}+Q\tau^2.
\end{array}
\end{equation}
We bound $1-(1-Q\tau^2)^{1/2}$ and $(1+Q\tau^2)^{1/2}-1$ by
$$\begin{array}{c}
\ds 1-(1-Q\tau^2)^{1/2}=\frac{Q\tau^2}{1+(1-Q\tau^2)^{1/2}}\leq Q\tau^2,\\[0.15in]
\ds (1+Q\tau^2)^{1/2}-1=\frac{Q\tau^2}{(1+Q\tau^2)^{1/2}+1}\leq Q\tau^2.
\end{array}  $$
Thus we complete the proof. \end{proof}

\subsection{Error estimates}
Define the errors 
$$e^x_n:=x(t_n)-x_n,~~e^v_{i,n}:=v_i(t_n)-v_{i,n},~~1\leq n\leq N_T,~~1\leq i\leq k.$$ 
We then perform the error estimates for the explicit Euler scheme (\ref{FDsadk}) of $k$-SD (\ref{sadk}) in the following theorem. 

\begin{theorem}\label{thmevk}
Suppose the Assumption A holds. Then the following estimate holds for $\tau$ sufficiently small 
$$\|e^x_{n}\|+\sum_{i=1}^k \|e^v_{i,n}\|\leq Q\tau,~~1\leq n\leq N_T. $$
Here $Q$ depends on $k$, $L$, $T$ and $\beta$ but is independent from $\tau$, $n$ and $N_T$.
\end{theorem}
\begin{proof} 
 Similar to the derivations in Theorem \ref{thmex} we could bound $e^x_n$ in terms of $e^v_{i,n}$ for $1\leq i\leq k$ as follows
\begin{equation}\label{MH}
\|e^x_{n}\|\leq Q\tau\sum_{m=1}^{n-1}\sum_{j=1}^k\|e^v_{j,m}\| +Q\tau ,~~1\leq n\leq N_T. 
\end{equation}
We subtract the second equation of (\ref{Refk}) from that of (\ref{FDsadk}) and apply
$v_i(t_n)-\tilde v_{i,n}$ as $(v_i(t_n)-v_{i,n})+(v_{i,n}-\tilde v_{i,n})=e^{v}_{i,n}+(v_{i,n}-\tilde v_{i,n})$
 to obtain
\begin{equation*}
\begin{array}{rl}
\ds e^v_{i,n}\hspace{-0.1in}&\ds=e^v_{i,n-1}+\tau\gamma\big(J(x(t_{n-1}))v_i(t_{n-1})-J(x_{n-1})v_{i,n-1}\big)\\[0.05in]
&\ds\quad~~-\tau\gamma \big[ v_i(t_{n-1})v_i(t_{n-1})^\top J(x(t_{n-1}))v_i(t_{n-1})\\[0.05in]
&\ds\qquad\quad-v_{i,n-1}v_{i,n-1}^\top J(x_{n-1})v_{i,n-1}\big]\\[0.05in]
&\ds\quad~~-2\tau\gamma \sum_{j=1}^{i-1}\big[ v_j(t_{n-1})v_j(t_{n-1})^\top J(x(t_{n-1}))v_i(t_{n-1})\\[0.05in]
&\ds\qquad\quad-v_{j,n-1}v_{j,n-1}^\top J(x_{n-1})v_{i,n-1}\big]-(v_{i,n}-\tilde v_{i,n})+\tau R_{n}^{v_i} ,
\end{array}
\end{equation*}
which, together with (\ref{MH}), Lemma \ref{lem2k} and the similar derivation as (\ref{mh7}), leads to
\begin{equation*}
\begin{array}{rl}
\ds\| e^v_{i,n}\|\hspace{-0.1in}&\ds\leq \|e^v_{i,n-1}\|+Q\tau\big[\|e^x_{n-1}\|+\|e^v_{i,n-1}\|\big]+Q\tau\sum_{j=1}^{i-1}\|e^v_{j,n-1}\| +Q\tau^2\\
\ds&\ds\leq \|e^v_{i,n-1}\|+Q\tau\|e^v_{i,n-1}\|+Q\tau^2\sum_{m=1}^{n-1}\sum_{j=1}^k\|e^v_{j,m}\|+Q\tau\sum_{j=1}^{i-1}\|e^v_{j,n-1}\| +Q\tau^2.
\end{array}
\end{equation*}
Adding this equation from $i=1$ to $k$ and denoting 
$$E^v_{n}:=\sum_{i=1}^k \|e^v_{i,n}\|,~~1\leq n\leq N_T $$
 yield an estimate in terms of $E^v_{n}$
\begin{equation}\label{mh8k}
\ds E^v_{n}\leq E^v_{n-1}+Q\tau E^v_{n-1}+Q\tau^2\sum_{m=1}^{n-1} E^v_{m} +Q\tau^2.
\end{equation}
Adding this equation from $n=1$ to $n_*$ and using 
$$\tau^2\sum_{n=1}^{n_*}\sum_{m=1}^{n-1}E^v_{m}=\tau^2\sum_{m=1}^{n_*-1}\sum_{n=m+1}^{n_*}E^v_{m}\leq T \tau\sum_{m=1}^{n_*-1}E^v_{m}$$
we get
\begin{equation}\label{mh9k}
\ds E^v_{n_*}\leq Q\tau\sum_{n=1}^{n_*-1}E^v_{n} +Q\tau.
\end{equation}
Then an application of the discrete Gronwall inequality leads to 
$$\ds E^v_{n}\leq Q\tau,~~1\leq n\leq N_T. $$
Plugging this estimate back to (\ref{MH}) yields the estimate of $\|e^x_n\|$ and we thus complete the proof.
\end{proof}

\section{Numerical experiments}
In this section, we carry out numerical experiments to substantiate the accuracy of the explicit Euler schemes (\ref{FDsad1}) and (\ref{FDsadk}) and the pathway convergence of the numerical solutions to the real searching pathway of HiSD. For the applications of these schemes, we refer \cite{YinSCM,YinSISC} for various physical examples and detailed discussions. We denote the following errors
$$\begin{array}{rl}
\ds\text{Err}(x)\hspace{-0.1in}&\ds:=\max_{1\leq n\leq N_T}\|x(t_n)-x_n\|,\\[0.1in]
\ds \text{Err}(v)\hspace{-0.1in}&\ds:=\max_{1\leq n\leq N_T}\|v(t_n)-v_n\|,\\[0.1in]
\ds\text{Err}(v_i)\hspace{-0.1in}&\ds:=\max_{1\leq n\leq N_T}\|v_i(t_n)-v_{i,n}\|,~~1\leq i\leq k
\end{array}  $$
and test their convergence rates. As the exact solutions to HiSD are not available, numerical solutions computed under $\tau=2^{-13}$ serve as the reference solutions for the first two examples, and we set $\beta=\gamma=T=1$ for simplicity.

\textbf{Example 1: Accuracy test under Minyaev-Quapp surface.}
We consider the saddle dynamics for the Minyaev-Quapp surface (\ref{MQS}) and compute its index-1 saddle point via scheme (\ref{FDsad1}) with the initial conditions $x(0)=(1,1)^\top$ and $v(0)=(0,1)^\top$ and index-2 saddle point via scheme (\ref{FDsadk}) with the initial conditions $x(0)=(1,1)^\top$, $v_1(0)=(0,1)^\top$ and $v_2(0)=(1,0)^\top$. Numerical results are presented in Tables \ref{table1:1}-\ref{table1:2}, which demonstrate the first-order accuracy of the explicit Euler schemes (\ref{FDsad1}) and (\ref{FDsadk}) as proved in Theorems \ref{thmev} and \ref{thmevk}.
\begin{table}[H]
\setlength{\abovecaptionskip}{0pt}
\centering
\caption{Convergence rates of scheme (\ref{FDsad1}) in Example 1.}
%\vspace{-0.25em}	
\begin{tabular}{ccccc} \cline{1-5}
$\tau$& Err($x$) & conv. rate &  Err($v$) &conv. rate\\ \cline{1-5}		
$1/32$  &2.19E-02	&	&1.72E-02	&\\
$1/64$	&1.03E-02&	1.09&	8.29E-03&	1.05\\
$1/128$	&4.95E-03&	1.05&	4.05E-03&	1.03\\
$1/256$	&2.40E-03&	1.04&	1.98E-03&	1.03\\
				\hline
				%			Pred.&	&	0.60&	&	1.00&	&	1.00\\
				%			\hline		
			\end{tabular}
			\label{table1:1}
		\end{table}

\begin{table}[H]
\setlength{\abovecaptionskip}{0pt}
\centering
\caption{Convergence rates of scheme (\ref{FDsadk}) in Example 1.}
\vspace{-0.25em}	
\begin{tabular}{ccccccc} \cline{1-7}
$\tau$& Err($x$) & conv. rate &  Err($v_1$) &conv. rate& Err($v_2$) &conv. rate\\ \cline{1-7}		
$1/32$&	1.50E-02&		&1.31E-02&		&1.31E-02&\\	
$1/64$&	7.41E-03&	1.02&	6.52E-03&	1.01&	6.52E-03&	1.01\\
$1/128$&	3.66E-03&	1.02&	3.23E-03&	1.01&	3.23E-03&	1.01\\
$1/256$&	1.79E-03&	1.03&	1.59E-03&	1.02&	1.59E-03&	1.02\\
				\hline
				%			Pred.&	&	0.60&	&	1.00&	&	1.00\\
				%			\hline		
			\end{tabular}
			\label{table1:2}
		\end{table}
		
\textbf{Example 2: Accuracy test under Eckhardt surface.}
We consider the saddle dynamics for the Eckhardt surface (\ref{ECS}) and compute its index-1 saddle point via scheme (\ref{FDsad1}) with the initial conditions 
$$x(0)=(-2,1)^\top,~~v(0)=\frac{1}{\sqrt{2}}(-1,1)^\top$$
 and index-2 saddle point via scheme (\ref{FDsadk}) with the initial conditions 
 $$x(0)=(-2,1)^\top,~~v_1(0)=\frac{1}{\sqrt{10}}(-1,3)^\top,~~v_2(0)=\frac{1}{\sqrt{10}}(3,1)^\top.$$
  Numerical results are presented in Tables \ref{table2:1}-\ref{table2:2}, which again show the first-order accuracy of the explicit Euler schemes (\ref{FDsad1}) and (\ref{FDsadk}) as proved in Theorems \ref{thmev} and \ref{thmevk}.

\begin{table}[H]
\setlength{\abovecaptionskip}{0pt}
\centering
\caption{Convergence rates of scheme (\ref{FDsad1}) in Example 2.}
\vspace{-0.25em}	
\begin{tabular}{ccccc} \cline{1-5}
$\tau$& Err($x$) & conv. rate &  Err($v$) &conv. rate\\ \cline{1-5}		
$1/32$&	1.41E-02	&	&2.16E-03	&\\
$1/64$	&6.98E-03&	1.01&	1.09E-03&	0.98\\
$1/128$&	3.45E-03&	1.01&	5.46E-04&	1.00\\
$1/256$&	1.70E-03&	1.02&	2.70E-04&	1.02\\
				\hline
				%			Pred.&	&	0.60&	&	1.00&	&	1.00\\
				%			\hline		
			\end{tabular}
			\label{table2:1}
		\end{table}

\begin{table}[H]
\setlength{\abovecaptionskip}{0pt}
\centering
\caption{Convergence rates of scheme (\ref{FDsadk}) in Example 2.}
\vspace{-0.25em}	
\begin{tabular}{ccccccc} \cline{1-7}
$\tau$& Err($x$) & conv. rate &  Err($v_1$) &conv. rate& Err($v_2$) &conv. rate\\ \cline{1-7}		
$1/32$&	5.78E-03&		&2.25E-03&		&2.25E-03&	\\
$1/64$&	2.86E-03&	1.02&	1.11E-03&	1.01&	1.11E-03&	1.01\\
$1/128$&	1.41E-03&	1.02&	5.51E-04&	1.02&	5.51E-04&	1.02\\
$1/256$&	6.95E-04&	1.03&	2.71E-04&	1.03&	2.71E-04&	1.03\\

				\hline
				%			Pred.&	&	0.60&	&	1.00&	&	1.00\\
				%			\hline		
			\end{tabular}
			\label{table2:2}
		\end{table}		

\textbf{Example 3: Convergence of dynamics under Eckhardt surface.}
In this example we aim to observe the numerical convergence of dynamics for HiSD. As the exact dynamics of HiSD is not available in practice, we approximate it by the numerical solutions under the very fine mesh size $\tau=2^{-8}$. The initial conditions are given as
$$x(0)=(1.5,1.2)^\top,~~v(0)=\frac{1}{\sqrt{5}}(-1,2)^\top.$$ 
Figure \ref{fig}(Left) indicates that the HiSD reaches a saddle point under both $T=10$ and $T=5$ and thus it suffices to take $T=5$. Figure \ref{fig}(Right) shows that not only the saddle point could be achieved under different step sizes, but the dynamics of HiSD globally converges to the exact one (i.e. the curve under $\tau=2^{-8}$) as the step size $\tau$ decreases. These observations justify the theoretical results and demonstrate that the proposed schemes are appropriate methods in computing the dynamic pathways for constructing the solution landscapes, e.g., \cite{HanYin,Xu_PRE,YinSCM}.
\begin{figure}[h!]
	\setlength{\abovecaptionskip}{0pt}
	\centering	\includegraphics[width=2.6in,height=2.2in]{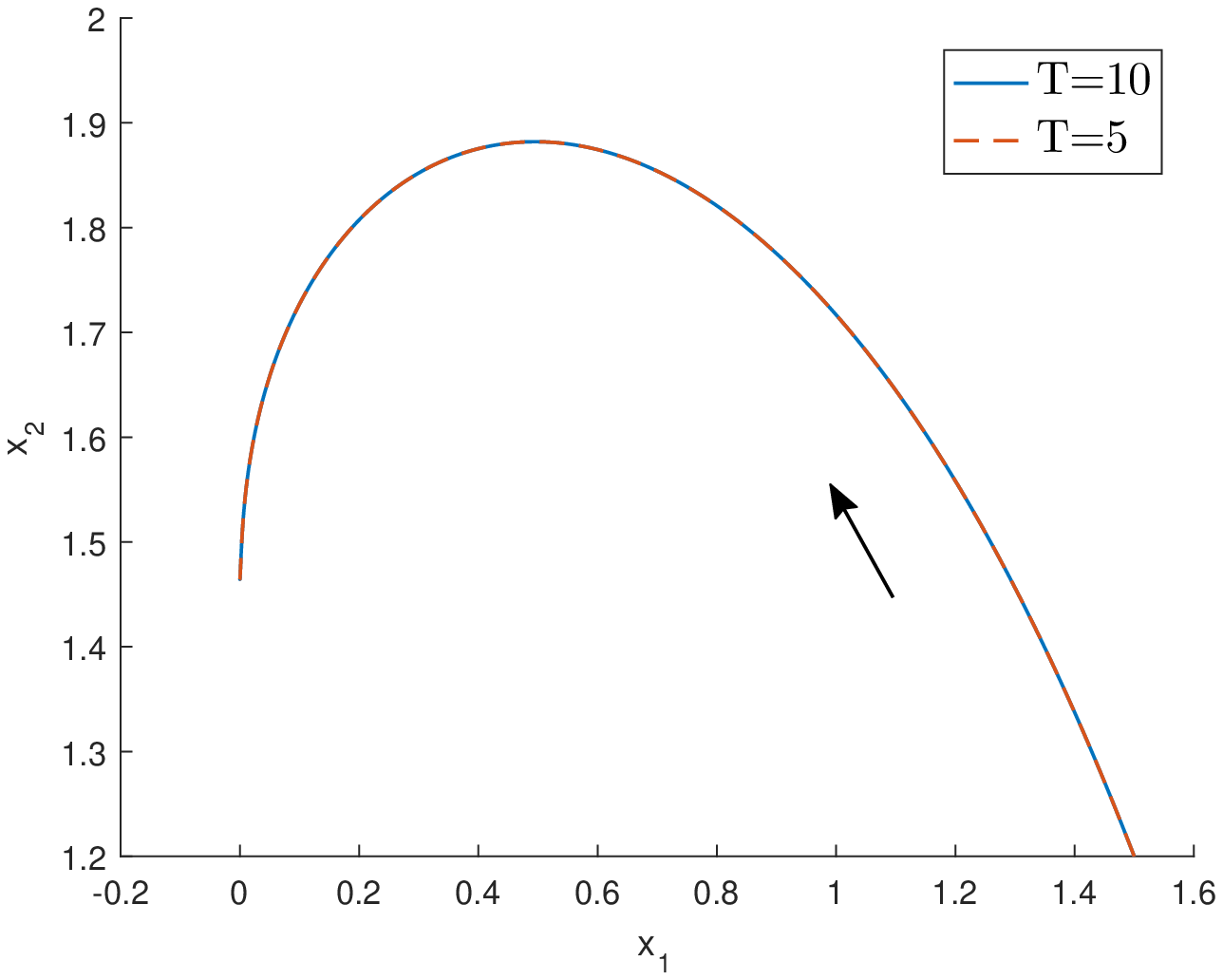}
\hspace{-0.2in}	\includegraphics[width=2.6in,height=2.2in]{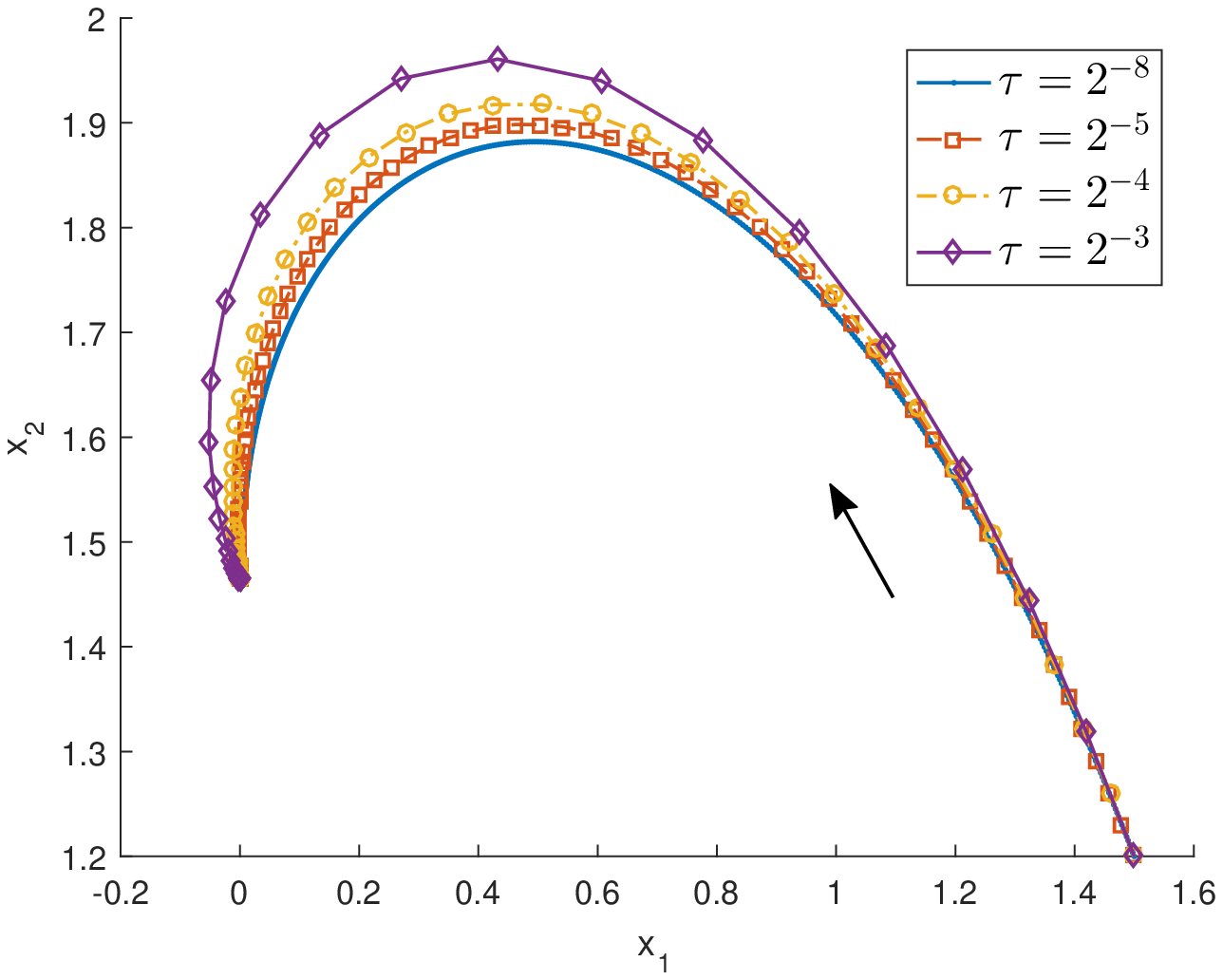}
	\caption{(Left) Numerical solution of $x(t)$ with $\tau=2^{-8}$ and different terminal time $T$; (Right) Numerical solution of $x(t)$ with $T=5$ and different  $\tau$. The symbols on the curves indicate the time steps.}
	\label{fig}
\end{figure}

\section{Extension to generalized high-index saddle dynamics for dynamical systems}
 In many autonomous dynamical systems $\dot x=F(x)$, the force $F(x)$ is not derived from the variation of some energy function/functional $E(x)$, which does not follow the formulations of the HiSD proposed in Section \ref{sec2}. In this case, the following GHiSD was developed in \cite{YinSCM} to compute the high-index saddle points for non-gradient systems:
\begin{equation}\label{ngsadk}
\left\{
\begin{array}{l}
\ds \frac{dx}{dt} =\bigg(I -2\sum_{j=1}^k v_jv_j^\top \bigg)F(x),\\[0.075in]
\ds \frac{dv_i}{dt}=(I-v_iv_i^\top)\tilde J(x)v_i -\sum_{j=1}^{i-1}v_jv_j^\top(\tilde J(x)+\tilde J^\top (x))v_i,~~1\leq i\leq k.
\end{array}
\right.
\end{equation}
 Here $\tilde J$ refers to the Jacobian of $F$, which is in general not symmetric. Compared with the high-index saddle dynamics (\ref{sadk}), which works for the gradient systems with the symmetric Hessian $J(x)$, $\tilde J(x)+\tilde J^\top (x)$ is used to replace $2J(x)$ in the dynamics of $\{v_i\}_{i=1}^k$. 
 
 Similar to (\ref{FDsadk}), we could propose the explicit Euler scheme of (\ref{ngsadk}) for $1\leq n\leq N_T$ as follows
 \begin{equation}\label{FDnsadk}
\left\{
\begin{array}{l}
\ds x_{n} =x_{n-1}+\tau\bigg(I -2\sum_{j=1}^k v_{j,n-1}v_{j,n-1}^\top \bigg)F(x_{n-1}),\\[0.1in]
\ds \tilde v_{i,n}=v_{i,n-1}+\tau\big( I-v_{i,n-1}v_{i,n-1}^\top\big)\tilde J(x_{n-1})v_{i,n-1}\\[0.1in]
\ds\hspace{0.7in}-\sum_{j=1}^{i-1}v_{j,n-1}v_{j,n-1}^\top\big(\tilde J(x_{n-1})+\tilde J^\top(x_{n-1})\big)v_{i,n-1},~~1\leq i\leq k,\\
\ds  v_{i,n}=\frac{1}{Y_{i,n}}\bigg(\ds\tilde v_{i,n}-\sum_{j=1}^{i-1}(\tilde v_{i,n}^\top v_{j,n})v_{j,n}\bigg),~~1\leq i\leq k.
\end{array}
\right.
\end{equation}

Intuitively, we may extend the developed methods and results in previous sections to GHiSD (\ref{ngsadk}) for non-gradient systems. Although there are differences between (\ref{sadk}) and (\ref{ngsadk}) as mentioned above, we notice that the derivations could be performed similarly. After carefully checking the preceding estimates, a key step lies in the quantification of $\tilde v_{m,n}^\top \tilde v_{i,n}$ in (\ref{MHH}). We thus calculate this product for $m<i$ for the case of generalized saddle dynamics (\ref{ngsadk}) as follows
\begin{equation*}
\begin{array}{l}
\ds \tilde v_{m,n}^\top \tilde v_{i,n}=\bigg[v_{m,n-1}+\tau\big( I-v_{m,n-1}v_{m,n-1}^\top\big)\tilde J(x_{n-1})v_{m,n-1}\\[0.15in]
\ds\qquad\qquad\qquad-\tau\sum_{j=1}^{m-1}v_{j,n-1}v_{j,n-1}^\top\big(\tilde J(x_{n-1})+\tilde J^\top(x_{n-1})\big)v_{m,n-1}\bigg]^\top\\[0.15in]
\ds\qquad\qquad\quad\bigg[v_{i,n-1}+\tau\big( I-v_{i,n-1}v_{i,n-1}^\top\big)\tilde J(x_{n-1})v_{i,n-1}\\[0.15in]
\ds\qquad\qquad\qquad-\sum_{j=1}^{i-1}v_{j,n-1}v_{j,n-1}^\top\big(\tilde J(x_{n-1})+\tilde J^\top(x_{n-1})\big)v_{i,n-1}\bigg]\\[0.15in]
\ds\qquad\quad~=\tau\big(v_{m,n-1}^\top \tilde J(x_{n-1}) v_{i,n-1}+v_{m,n-1}^\top \tilde J^\top(x_{n-1})v_{i,n-1}\\[0.1in]
\ds\qquad\qquad\qquad-v_{m,n-1}^\top \big(\tilde J(x_{n-1})+\tilde J^\top(x_{n-1})\big)v_{i,n-1}\big)+\tau^2[\cdots]=\tau^2[\cdots],
\end{array}
\end{equation*}
where we used the fact that
$$\begin{array}{l}
\ds v_{m,n-1}^\top \tilde J(x_{n-1}) v_{i,n-1}+v_{m,n-1}^\top \tilde J^\top(x_{n-1})v_{i,n-1}\\[0.05in]
\ds\qquad\qquad\qquad-v_{m,n-1}^\top \big(\tilde J(x_{n-1})+\tilde J^\top(x_{n-1})\big)v_{i,n-1}=0.
\end{array}  $$
Note that this nice property stems from the symmetrization $\tilde J(x)+\tilde J^\top(x)$ in the GHiSD (\ref{ngsadk}). By virtue of this symmetrization, we could derive the error estimates for explicit Euler scheme (\ref{FDnsadk}) in parallel with the proofs in previous sections, and thus show the numerical accuracy of GHiSD (\ref{ngsadk}) for non-gradient systems. 
 
\section{Conclusions}

In this paper we develop novel techniques to overcome the strong nonlinearity and the difficulties caused by the orthonormalization procedure to prove error estimates for Euler discretization of HiSD with respect to the time step size, which serves as an important compensation for the gap between the continuous problem and the numerical implementation. The developed methods and results are further extended to study the GHiSD for non-gradient systems, and thus provide theoretical supports for the accuracy of numerical computations. 

There are several other potential extensions of the current work. For instance, the ideas and techniques could be employed and improved to analyze implicit schemes or high-order methods like the  backward difference formulas  of HiSD in order to retain high-order accuracy. How to relax the linearly growth and Lipschitz conditions on the force $F$ and the Hessian $J$ (or Jacobian in non-gradient systems) in the error estimates is also an important but challenging topic as many energy surfaces do not enjoy these properties. We may further extend the methods to study the numerical schemes of the dimer-type methods such as the shrinking dimer dynamics \cite{ZhaDu}. However, as the multiplication of the Hessian and the vector is approximated by dimer methods in order to reduce the computational costs, the current Hessian-based proofs do not apply and additional investigations will be carried out in the near future.
Recently, a constrained high-index saddle dynamics (CHiSD) was developed to compute high-index saddle points of an energy functional subject to equality constraints \cite{CHiSD2021}. Applying the Riemannian gradients and Hessians, the CHiSD is derived as the dynamical system with a transformed gradient flow. In the numerical implementation, the retraction operator and vector transport are introduced to discretize the CHiSD. Thus, it will be very interesting to study the error estimates for numerical schemes of CHiSD in future.

It is also worth mentioning that in the current work we focus on the dynamics of HiSD on the finite interval $t\in [0,T]$ and prove error estimates to show the pathway convergence of numerical solutions of HiSD. For most applications, numerical schemes of HiSD could reach the target saddle point within a certain number of iterations (e.g., the observations in the left plot of Figure \ref{fig}), and thus the derived error estimate results in this work could apply by setting a suitable $T$. Nevertheless, the convergence to the saddle point is theoretically determined by the rate of $x(t)\rightarrow x_*$ as $t\rightarrow \infty$, thus it is desirable to perform the error estimate of $x_n-x_*$ for $n\rightarrow \infty$ as in the literature of optimization algorithms \cite{LiWan,Nes}. In a very recent work \cite{Luo} the convergence analysis of HiSD is performed to determine the convergence rate of $x_n-x_*$ for $n\rightarrow \infty$. It should be pointed out that the result in the current work that $v_{1,n}-v_1(t_n)$ has the first-order accuracy with respect to the time step size was borrowed in \cite{Luo} in order to prove the approximation of eigenvector, which again shows the importance of novel techniques proposed in this work.

\section*{Acknowledgements}
This work was partially supported by the National Key Research and Development Program of China  2021YFF1200500; the National Natural Science Foundation of China No.~12050002 and 21790340; the International Postdoctoral Exchange Fellowship Program No.~YJ20210019; the China Postdoctoral Science Foundation Nos.~2021TQ0017 and 2021M700244.

\end{document}